\documentclass{amsart}

\newtheorem{theorem}{Theorem}[section]
\newtheorem{lemma}[theorem]{Lemma}

\theoremstyle{definition}
\newtheorem{definition}[theorem]{Definition}
\newtheorem{example}[theorem]{Example}

\newtheorem{corollary}[theorem]{Corollary}

\theoremstyle{remark}

\numberwithin{equation}{section}

\usepackage{amsmath, amsthm, amscd, amsfonts, amssymb, graphicx, color}
\usepackage[bookmarksnumbered, colorlinks, plainpages]{hyperref}
\usepackage{leftidx}
\usepackage{mathtools}
\usepackage{enumerate}
\usepackage{caption,epsfig}
\usepackage{ graphicx}
\usepackage{tabularx}
\usepackage{booktabs}
\usepackage{graphicx}

\begin{document}

\title{Bigeometric Ces$\grave{\text{A}}$ro difference sequence spaces and Hermite interpolation}

%    Remove any unused author tags.

%    author one information
\author{Sanjay Kumar mahto}
\address{Department of Mathematics, Indian Institute of Technology Kharagpur, Kharagpur 721302, India}
\curraddr{}
\email{kumarmahtosanjay@iitkgp.ac.in, skmahto0777@gmail.com}
\thanks{}

%author two information
\author{Atanu Manna}
\address{Indian Institute of Carpet Technology, Chauri Road, Bhadohi 221401, Uttar Pradesh, India}
\curraddr{}
\email{ atanuiitkgp86@gmail.com}
\thanks{}

%author three information
\author{P. D. Srivastava}
\address{Department of Mathematics, Indian Institute of Technology Bhilai, Raipur, Chattisgarh 492015, India}
\curraddr{}
\email{ pds@maths.iitkgp.ac.in}
\thanks{}



\date{}

\dedicatory{}

\begin{abstract}
In this paper, we introduce some difference sequence spaces in bigeometric calculus. We determine the $\alpha$-duals of these sequence spaces and study their matrix transformations. We also develop an interpolating polynomial in bigeometric calculus which is analogous to the classical Hermite interpolating polynomial.
\end{abstract}

\maketitle

\section{Introduction}
 Bigeometric calculus is one of the non-Newtonian calculi developed by Grossman et al. \cite{grossman1972non,grossman1983bigeometric}
during the years 1967-1983. In bigeometric calculus, changes and accumulations in arguments and values of a function are measured by ratios and products respectively, whereas in classical calculus, changes are measured by differences and accumulations are measured by sums. The bigeometric calculus may be considered as a byproduct of ambiguity among scholars in choosing either differences or ratios for the estimation of deviations. Galileo observed that the ratios are more convenient in measuring deviations.\par
The important applications of bigeometric calculus are seen in fractal dynamics of materials \cite{rybaczuk2000fractal}, fractal dynamics of biological systems \cite{rybaczuk1999critical}, etc. Moreover, Multiplicative calculus is used to establish non-Newtonian Runge-Kutta methods \cite{aniszewska2007multiplicative}, Lorenz systems \cite{aniszewska2005analysis}, and some finite difference methods \cite{riza2009multiplicative}. Some non-Newtonian Hilbert spaces \cite{kadak2014construction} are constructed by Kadak et al. Some non-Newtonian metric spaces and its applications can be seen in
\cite{bashirov2008multiplicative,dovsenovic2016multiplicative}. \c{C}akmak and Ba\c{s}ar \cite{ccakmak2012some} have constructed the non-Newtonian real field $\mathbb{R}(N)$ and defined the sequence spaces $w(N), l_{\infty}(N), c(N), c_0(N)$ and $l_{p}(N)$ in $\mathbb{R}(N)$. The matrix transformations between these spaces are also studied by them\cite{ccakmak2015some}. While Kadak
\cite{kadak2014determination,kadak2015classical,kadak2016cesaro,kadak2016multiplicative}, in a series of papers, has made a significant contribution in constructing non-Newtonian sequence spaces and in studying their K$\ddot{\text{o}}$the-Toeplitz duals and matrix transformations.\par
The classical difference sequence spaces are first introduced by Kizmaz \cite{kizmaz1981certain}. He has studied the spaces $Z(\Delta)$ for $Z = l_{\infty}$, $c$ or $c_0$ which are defined as
\begin{equation}\label{kizmazspaces}
Z(\Delta) = \{ x = (x_k): (\Delta x_k) \in Z\},
\end{equation}
where $\Delta x_k = x_k - x_{k+1}$. Later on, Et and \c{C}olak \cite{et1995some} have generalized these spaces by replacing the first order difference operator with $m$-th order difference operator and hence defined the spaces $Z(\Delta^m)$ for $Z = l_{\infty}$, $c$ or $c_0$ as follows:
\begin{equation}\label{etspaces}
Z(\Delta^m) = \{ x = (x_k): (\Delta^m x_k) \in Z\},
\end{equation}
where $\Delta^m x_k = \Delta^{m-1}(x_k - x_{k+1})$ $= \sum\limits_{v=0}^{m}(-1)^v {\binom mv} x_{k+v}$. Following these spaces Orhan \cite{orhan1983cesaro} has studied the Ces\`{a}ro difference sequence spaces $C_p(\Delta)$ and $C_{\infty}(\Delta)$ for $1\leq p < \infty$. Subsequently, Et \cite{et2000some} has used the $m$-th order difference operator instead of first order difference operator in the spaces of Orhan and constructed the spaces
\begin{equation}\label{etcesaro1}
C_p(\Delta^m) = \left\{ x = (x_k): \sum\limits_{n=1}^{\infty} \left\lvert \frac{1}{n}\sum\limits_{k=1}^{n} \Delta^m x_k\right\rvert^p < \infty\right\}
\end{equation}
and
\begin{equation}\label{etcesaro2}
C_{\infty}(\Delta^m) = \left\{ x = (x_k): \sup_n \left\lvert \frac{1}{n} \sum\limits_{k=1}^{n} \Delta^m x_k\right\rvert < \infty\right\}.
\end{equation}
 Recently, Boruah and Hazarika \cite{boruah2017application} have defined the geometric difference sequence spaces $X(\Delta_{G})$ for $X= l_{\infty}^G$, $c^G$ or $c_0^G$ which are analogous to the spaces of Kizmaz \eqref{kizmazspaces}. They have studied duals of these spaces and constructed geometric Newton's forward and backward interpolation formulae.\par
In this paper, we have introduced some Ces$\grave{\text{a}}$ro difference sequence spaces in bigeometric calculus analogous to the classical sequence spaces as defined in \eqref{etcesaro1} and \eqref{etcesaro2}. We have also studied bigeometric $\alpha$-duals and matrix transformations of these spaces and formulated the Hermite interpolation formula in bigeometric calculus.
%==========================================================================================================================

\section{Preliminaries}
\begin{definition}(Arithmetic system)\cite{grossman1972non,grossman1983bigeometric}:
An arithmetic system consists of a set $R$ with four operations namely addition, subtraction, multiplication, division and an ordering relation that satisfy the axioms of a complete ordered field. The set $R$ is called realm, and the members of the set $R$ are called numbers of the system. The set of all real numbers $\mathbb{R}$ with the usual $+, -, \times, /$ and $<$ is known as the classical arithmetic system.
 \end{definition}

 \begin{definition}(Generator)\cite{grossman1972non,grossman1983bigeometric}:
 A one-to-one function $\phi:\mathbb{R} \rightarrow \mathbb{R}$ is called a generator whose range $A$ is a subset of $\mathbb{R}$.
 \end{definition}

 \begin{definition}($\phi$-arithmetic)\cite{grossman1972non,grossman1983bigeometric}:
 The $\phi$-arithmetic with generator $\phi$ is an arithmetic system whose realm is the range $A$ of $\phi$ and the operations $\phi$-addition `$+_{\phi}$,' $\phi$-subtraction `$-_{\phi}$,' $\phi$-multiplication`$\times _{\phi}$,' $\phi$-division `$/_{\phi}$' and $\phi$-order `$<_{\phi}$' are defined as follows:
\allowdisplaybreaks
\begin{align*}
&u +_{\phi} v = \phi\{\phi^{-1}(u) + \phi^{-1}(v)\},\\
&u -_{\phi} v = \phi\{\phi^{-1}(u) - \phi^{-1}(v)\},\\
&u \times _{\phi} v = \phi\{\phi^{-1}(u) \times \phi^{-1}(v)\},\\
&u\  / _{\phi} \ v = \phi\{\phi^{-1}(u) / \phi^{-1}(v)\}\\
\text{and} \ \ &u < _{\phi} v  \ \text{if and only if} \ \phi^{-1}(u) < \phi^{-1}(v).\\
\end{align*}
For example, the classical arithmetic is generated by the identity function.
\end{definition}

Every ordered pair of arithmetics ($\phi$-arithmetic, $\phi '$-arithmetic) gives rise to a calculus, where $\phi$-arithmetic is used for arguments and $\phi '$-arithmetic is used for values of functions in the calculus. For a particular choice of $\phi$ and $\phi '$ the following calculi can be generated:\\
\begin{center}
\begin{tabular}{c| c c}
Calculus & $\phi$ & $\phi '$ \\
\hline
classical & I & I \\
geometric &  I & $\exp$ \\
anageometric & $\exp$ & I \\
bigeometric & $\exp $ & $\exp$
\end{tabular}
\end{center}
\vspace{2mm}
where $I$ and $\exp$ denote the identity and exponential functions respectively.
\begin{definition}(Geometric arithmetic)\cite{grossman1972non,grossman1983bigeometric}:
 The arithmetic generated by the exponential function is called geometric arithmetic.
 \end{definition}

 \begin{definition}(Geometric real number)\cite{turkmen2012some}:
 The realm of the geometric arithmetic is the set of all positive real numbers which we denote by $\mathbb{R}(G)$ and call it the set of geometric real numbers. Then,
\begin{equation}
\mathbb{R}(G) = \{e^x : x \in \mathbb{R}\}.
\end{equation}
\end{definition}
\subsection{Some properties of geometric arithmetic system:}
 Let us denote the geometric operations addition, subtraction, multiplication, and division by $\oplus, \ominus, \odot$, and $ \oslash$ respectively and the ordering relation by the usual symbol $<$. Some properties of the geometric arithmetic system are as follows\cite{grossman1972non, boruah2017application}: For all $u, v \in \mathbb{R}(G)$, we have
\begin{enumerate}[(i)]
\item $(\mathbb{R}(G), \oplus , \odot)$ is a field with geometric identity $e$ and geometric zero 1.
  \item Geometric addition: $u \oplus v = \exp \{ \ln u + \ln v\} = uv$.
  \item Geometric subtraction: $u \ominus v = \exp \{\ln u - \ln v\} = u/v$.
  \item Geometric multiplication: $u \odot v = \exp \{ \ln u \ln v\}$ = $u^{\ln v}$.
  \item Geometric division: $u \oslash v$ or $\frac{u}{v}{}_{\tiny{G}}= \exp \{ \ln u / \ln v\} = u^{\frac{1}{\ln v}}, v \neq 1$.
  \item Since $u < v$ if and only if $\ln u < \ln v$, the usual $<$ relation is taken as geometric ordering relation.
  \item Geometric exponentiation: For a real number $q$, we have $u^{q_G} = \exp(\ln u)^q = u^{\ln ^{q-1} u}$.
  \item $e^n \odot u =\underbrace{u \oplus u \oplus \dots \oplus u}_{n \ \text{times}} = u^n$.
  \item Geometric modulus: $\lvert u \lvert _G = \exp\lvert \ln u \rvert = \begin{cases} u, \ \text{when} \ u > 1 \\ 1, \ \text{when} \ u =1 \\ \frac{1}{u}, \ \text{when} \ 0< u < 1 \end{cases}$. \\ Thus, $\lvert u \lvert _G$ is always greater than or equal to one.
  \item $\lvert u \rvert_G^{q_G}\geq 1$ for all $u \in \mathbb{R}(G)$ and $q \in \mathbb{R}$.
  \item $\lvert u \odot v \rvert _G = \lvert u \lvert _G \odot \lvert v \rvert _G$.
  \item $ \lvert u \oplus v\rvert _G \leq \lvert u \rvert_G \oplus \lvert v \rvert _G$ (Triangular inequality).
   \item The symbols $\leftidx{_G}\sum$ and $\leftidx{_G}\prod$ represent the geometric sum and geometric product of sequence of numbers respectively.
\end{enumerate}

\begin{definition}(Geometric normed space):
Let $X$ be a linear space over $\mathbb{R}(G)$. Then $X$ is said to be a normed space if there exists a function $\lVert \cdot \rVert_G$ from $X$ to $\mathbb{R}(G)$ such that for all $x$, $y \in X$ and $a\in \mathbb{R}(G)$,
\begin{enumerate}
  \item $\lVert x \rVert_G \geq 1$ .
  \item $\lVert x \rVert_G =1$ if and only if $x = \theta$, where $\theta$ is the zero element of $X$ .
  \item $\lVert a\odot x \rVert_G = \lvert a \rvert_G\odot \lVert x \rVert_G$.
  \item $\lVert x \oplus y \rVert_G \leq \lVert x \rVert_G \oplus \lVert y \rVert_G$.
\end{enumerate}
\end{definition}
\begin{definition}(Bigeometric derivative)\cite{grossman1972non,grossman1983bigeometric}:
The bigeometric derivative of a function $f:\mathbb{R}(G)\rightarrow\mathbb{R}(G)$ at a point $a \in \mathbb{R}(G)$ is given by
\begin{equation}
[D_G f](a) = \lim\limits_{x \rightarrow a} \left[ \frac{f(x)}{f(a)}\right]^{\frac{1}{\ln x - \ln a}}.
\end{equation}
\end{definition}

\subsection{Some useful results of bigeometric derivative:}
\begin{enumerate}
  \item Let $f'(a)$ be the classical derivative of a function $f(x)$ at the point $a$; then the bigeometric derivative of the function at the same point is given by
\begin{equation}
[D_G f](a) = e ^{\frac{a f'(a)}{f(a)}}.
\end{equation}
\item Bigeometric derivative of some of the important functions are as follows:
\begin{itemize}
  \item $D_G e^x = e^x \ ; \ x > 0.$
  \item $D_G \ln (x) = e^{\frac{1}{\ln (x)}} \ ; \ x > 1.$
  \item $D_G \sin(x) = e^{x \cot (x)} \ ; \ x \in \cup_{n=0}^{\infty}(2n\pi , (2n+1)\pi)$.
\end{itemize}
\item Let $f$ and $g$ be two functions which are defined in $\mathbb{R}(G)$ and whose ranges are subsets of $\mathbb{R}(G)$; then we have
\begin{itemize}
  \item $D_G(f(x)\oplus g(x)) = D_G f(x) \oplus D_G g(x)$.
  \item $D_G(f(x)\ominus g(x)) = D_G f(x) \ominus D_G g(x)$.
  \item $D_G (f(x)\odot g(x)) = (D_G f(x))\odot g(x) \oplus (D_G g(x))\odot f(x)$.
  \item $D_G (f(x) \oslash g(x)) = \frac{(D_G f(x))\odot g(x) \ominus (D_G g(x))\odot f(x)}{g(x)\odot g(x)}{}_G$.
\end{itemize}
\end{enumerate}
%%=================================================================================================================
\begin{theorem}\cite{ccakmak2012some}
The ordered pair $(\mathbb{R}(G), d_G)$ is a complete metric space with respect to the metric $d_G$ defined by
\begin{equation*}
d_G(x,y) = \lvert x \ominus y\rvert_G
\end{equation*}
for all $x, y \in \mathbb{R}(G)$.
\end{theorem}
%%==================================================================================================================
\begin{lemma}\label{lemma_maddox}\cite{maddox1967spaces}
For $a, b \in \mathbb{R}$ and $p \in [1, \infty)$, the following inequality holds:
\begin{equation*}
\lvert a + b\rvert^p \leq 2^{p-1}(\lvert a \rvert^p + \lvert b \rvert^p).
\end{equation*}
\end{lemma}
%%===================================================================================================================
\begin{lemma}(Jessen's inequality)\label{lemma_pq}\cite{et2000some}:
Let $(a_k)$ be a sequence in $\mathbb{R}$ and $0< p < q$, then the following inequality holds:
\begin{equation*}
\left( \sum\limits_{k=1}^{n}\lvert a_k\rvert^q\right)^{1/q} \leq \left( \sum\limits_{k=1}^{n}\lvert a_k\rvert^p\right)^{1/p}.
\end{equation*}
\end{lemma}
%%================================================================================================================
\begin{lemma}(Geometric Minkowski's inequality)\cite{turkmen2012some}\label{lemma_minkowski}:
Let $p\geq 1$ and $a_k$, $b_k \in \mathbb{R}(G)$ for all $k\in \mathbb{N}$. Then,
\begin{equation*}
\left( \leftidx{_G}\sum\limits_{k=1}^{\infty} \lvert a_k \oplus b_k \rvert_G^{p_G}\right)^{(1/p)_G} \leq \left( \leftidx{_G}\sum\limits_{k=1}^{\infty} \lvert a_k \rvert_G^{p_G}\right)^{(1/p)_G} \oplus \left( \leftidx{_G}\sum\limits_{k=1}^{\infty} \lvert b_k \rvert_G^{p_G}\right)^{(1/p)_G} .
\end{equation*}
\end{lemma}
%%===================================================================================================================

\begin{lemma}\cite{boruah2017application}\label{lemma_hazarika}
$\sup\limits_{k} \lvert x_k \ominus x_{k+1}\rvert_G< \infty$ if and only if  \emph{(i)} $ \sup\limits_k e^{k^{-1}}\odot \lvert x_k \rvert_G < \infty$ and \emph{(ii)} $\sup\limits_k \left\lvert x_k \ominus e^{k(k+1)^{-1}}\odot x_{k+1} \right\rvert_G < \infty$ hold.
\end{lemma}
%%=====================================================================================================================

Putting $\Delta_G^{m-1}x_k$ instead of $x_k$ in the above lemma, we get the following result.
\begin{corollary}\cite{khan55generalized}\label{lemma_khanequivalent}
The following assertions $(i)$ and $(ii)$ are equivalent.
\begin{enumerate}[(i)]
  \item $\sup\limits_k \lvert \Delta_G^{m-1} x_k \ominus \Delta_G^{m-1}x_{k+1}\rvert_G < \infty$.
  \item \begin{enumerate}[(a)]
          \item$ \sup\limits_{k} e^{k^{-1}}\odot \lvert \Delta_G^{m-1} x_k\rvert_G < \infty$.
          \item $\sup\limits_k \lvert \Delta_G^{m-1}x_k \ominus e^{k(k+1)^{-1}} \odot \Delta_G^{m-1}x_{k+1}\rvert < \infty$.
        \end{enumerate}
\end{enumerate}
\end{corollary}
%%================================================================================================================

\begin{lemma}\cite{khan55generalized}\label{lemma_khanekminusi}
If $\sup\limits_k e^{k^{-i}}\odot\lvert \Delta_G^{m-1} x_k\rvert_G < \infty $, then $\sup\limits_k e^{k^{-(i+1)}}\odot \lvert \Delta_G^{m-(i+1)}x_k\rvert_G<\infty$ for all $i, m \in \mathbb{N}$ and $1\leq i < m$.
\end{lemma}

\begin{corollary}\cite{khan55generalized}\label{corollary_khanekminus1}
If $\sup\limits_k e^{k^{-1}}\odot \lvert \Delta_G^{m-1}x_k\rvert_G < \infty$, then $\sup\limits_k e^{k^{-m}}\odot\lvert x_k \rvert_G < \infty$.
\end{corollary}
%%==================================================================================================================

\section{Main results}
 In this section, we define some bigeometric Ces$\grave{\text{a}}$ro difference sequence spaces.  Let $(x_k)$ be a sequence in the set $\mathbb{R}(G)$ of geometric real numbers; then the first order geometric difference operator $\Delta_G$ is defined by $\Delta_G x_k = x_k \ominus x_{k+1}$ and the $m$-th order geometric difference operator $\Delta_G^m$ is defined by $\Delta_G^m x_k = \Delta_G^{m-1}(x_k\ominus x_{k+1})$. Thus $\Delta_G^m x_k = \leftidx{_G}\sum\limits_{v=0}^{m} (\ominus e)^{v_G}\odot e^{\binom mv}\odot x_{k+v}$. It is easy to verify that the geometric difference operator is a linear operator. Let $w(G)$ denotes the set of all sequences in $\mathbb{R}(G)$. Then the set $w(G)$ is a linear space over $R(G)$ with respect to the operations \emph{(a)} vector addition
 \begin{align*}
 &\oplus : w(G)\times w(G) \rightarrow w(G) \\
\text{ defined by} \ &((x_k), (y_k)) \rightarrow (x_k)\oplus (y_k) = (x_k \oplus y_k) \ \text{and}
 \end{align*}
  \emph{(b)} scalar multiplication
 \begin{align*}
& \odot : \mathbb{R}(G)\times w(G) \rightarrow w(G) \\
 \text{ defined by} \ &((x_k), (y_k)) \rightarrow (x_k)\odot (y_k) = (x_k \odot y_k).
 \end{align*}
 We introduce the following sequence spaces in bigeometric calculus as follows:
\begin{equation}\label{myspace1}
C_p^G(\Delta_G^m) = \left\{x = (x_k)\in w(G) :\leftidx{_G}\sum\limits_{n=1}^{\infty}\left \lvert (e\oslash e^n)\odot  \leftidx{_G}\sum\limits_{k=1}^{n}\Delta_G^m x_k\right\rvert_G^{p_G} < \infty \right\}
\end{equation}
for $1\leq p < \infty$ and
\begin{equation}\label{myspace2}
C_{\infty}^G(\Delta_G^m) = \left\{x = (x_k)\in w(G) :\sup\limits_{n}\left \lvert (e\oslash e^n)\odot  \leftidx{_G}\sum\limits_{k=1}^{n}\Delta_G^m x_k\right\rvert_G < \infty \right\}.
\end{equation}
%where $\Delta_G^m x_k = \leftidx{_G}\sum\limits_{v=0}^{m} (\ominus e)^{v_G}\odot e^{\binom mv}\odot x_{k+v}$.
%==========================================================================================

 \begin{lemma}\label{lemma_maddoxG}
For $a, b \in \mathbb{R(G)}$ and $p \in [1, \infty)$, the following inequality holds:
\begin{equation*}
\left\lvert a \oplus b\right\rvert_G^{p_G} \leq e^{2^{p-1}}\odot \left(\lvert a \rvert_G^{p_G} \oplus \lvert b \rvert_G^{p_G}\right).
\end{equation*}
\end{lemma}
\begin{proof}
Let $u = \left\lvert a \oplus b\right\rvert_G^{p_G}$; then
\begin{equation*}
   u = \lvert ab\rvert_G^{p_G}
\end{equation*}
Using the definition of geometric modulus, we have
\begin{equation*}
    u = \left\{\exp\lvert \ln(ab) \rvert\right\}^{p_G}
\end{equation*}
Again using the definition of geometric exponentiation, we get
\begin{eqnarray}
    u &=& \exp\left[\ln\left\{\exp\lvert \ln(ab) \rvert\right\}\right]^p\nonumber\\
    &=&\exp\lvert \ln(ab) \rvert^p\nonumber\\
\ln u &=& \lvert \ln a + \ln b \rvert^p \label{equation_gmaddox1}.
    \end{eqnarray}
By applying Lemma \ref{lemma_maddox} in \eqref{equation_gmaddox1}, we get
\allowdisplaybreaks
\begin{eqnarray}
\ln u &\leq& 2^{p-1}\left(\lvert \ln a \rvert^p + \lvert \ln b \rvert^p \right)\nonumber\\
         &=& 2^{p-1}\left[\ln \exp\left\{\ln (\exp\lvert\ln a \rvert)\right\}^p + \ln \exp\left\{\ln (\exp\lvert\ln b \rvert)\right\}^p \right] \nonumber
\end{eqnarray}
Using the definition of geometric modulus, we have
\begin{equation*}
\ln u = 2^{p-1}\left[\ln \exp\left\{\ln \lvert a \rvert_G\right\}^p + \ln \exp\left\{\ln \lvert b \rvert_G\right\}^p \right]
\end{equation*}
Again using the definition of geometric exponentation, we get
\begin{eqnarray*}
       \ln u   &=& 2^{p-1}\left[ \ln \lvert a\rvert _G^{p_G} + \ln \lvert b\rvert _G^{p_G}\right] \\
          &=&  2^{p-1} \ln \left(\lvert a\rvert _G^{p_G} \lvert b\rvert _G^{p_G} \right)\\
          &=&  2^{p-1} \ln \left(\lvert a\rvert _G^{p_G} \oplus\lvert b\rvert _G^{p_G}\right).
\end{eqnarray*}
That is,
\begin{equation*}
 u \leq  e^{2^{p-1} \ln \left(\lvert a\rvert _G^{p_G} \oplus\lvert b\rvert _G^{p_G} \right)}.
\end{equation*}
Thus, we have
\begin{equation*}
\left\lvert a \oplus b\right\rvert_G^{p_G}  \leq e^{ 2^{p-1}} \odot \left(\lvert a\rvert _G^{p_G} \oplus\lvert b\rvert _G^{p_G} \right).
\end{equation*}
\end{proof}
%==========================================================

\begin{theorem}
The sets $C_p^G(\Delta_G^m)$ and $C_{\infty}^G(\Delta_G^m)$ are linear subspaces of $w(G)$.
\end{theorem}
\begin{proof}
As the geometric difference operator $\Delta_G^m$ is linear, then using Lemma \ref{lemma_maddoxG} it is easy to prove that the sets $C_p^G(\Delta_G^m)$ and $C_{\infty}^G(\Delta_G^m)$ are linear subspaces of $w(G)$.
 \end{proof}
%%====================================================================================================================

\begin{theorem}
The linear spaces $C_p^G(\Delta_G^m)$ and $C_{\infty}^G(\Delta_G^m)$ are normed spaces with respect to the norms
\begin{equation}\label{norm1}
\lVert x \rVert_p^G = \leftidx{_G}\sum\limits_{i=1}^{m} \lvert x_i \rvert_G \oplus \left(  \leftidx{_G}\sum\limits_{n=1}^{\infty} \left\lvert(e\oslash e^n) \odot \leftidx{_G}\sum\limits_{k=1}^{n}\Delta_G^m x_k\right\rvert_G^{p_G}\right)^{(1/p)_G}
\end{equation}
and
\begin{equation}\label{norm2}
\lVert x\rVert_{\infty}^{G} = \leftidx{_G}\sum\limits_{i=1}^{m} \lvert x_i \rvert_G \oplus \sup\limits_{n}\left\lvert (e\oslash e^n)\odot \leftidx{_G}\sum\limits_{k=1}^{n}\Delta_G^m x_k\right\rvert_G
\end{equation}
respectively.
\end{theorem}
\begin{proof}
Here we prove the theorem for the space $(C_p^G(\Delta_G^m), \lvert x \rvert_p^G)$ leaving the proof of other space as the proof runs on the parallel lines. Let $x =(x_k)$, $y = (y_k) \in C_p^G(\Delta_G^m)$ and $a\in \mathbb{R}(G)$.\par
We first show that $\lVert x \rVert_p^G \geq 1$ for all $x\in C_p^G(\Delta_G^m)$. We know that geometric modulus is always greater than or equal to one, so
\begin{align*}
&\lvert x_i \rvert_G \geq 1 \ \text{for all} \ i = 1, 2, \dots, m \\
\text{and} \ &\left\lvert (e\oslash e^n)\odot \leftidx{_G}\sum\limits_{k=1}^{n}\Delta_G^m x_k \right\rvert_G \geq 1 \ \text{for all} \ n.
\end{align*}
Using the property 2.1(x) of geometric arithmetic and taking geometric summations, we get
\begin{equation}\label{twopartsgreaterthan1}
\leftidx{_G}\sum\limits_{i=1}^{m} \lvert x_i \rvert_G \geq 1 \ \text{and} \ \left(\leftidx{_G}\sum\limits_{n=1}^{\infty}\left\lvert (e\oslash e^n)\odot \leftidx{_G}\sum\limits_{k=1}^{n}\Delta_G^m x_k \right\rvert_G ^ {p_G}\right)^{(1/p)_G} \geq 1.
\end{equation}
That is,
\begin{equation}\label{normgreaterthanone}
\lVert x \rVert_p^G = \leftidx{_G}\sum\limits_{i=1}^{m} \lvert x_i \rvert_G \oplus \left(\leftidx{_G}\sum\limits_{n=1}^{\infty}\left\lvert (e\oslash e^n)\odot \leftidx{_G}\sum\limits_{k=1}^{n}\Delta_G^m x_k \right\rvert_G ^ {p_G}\right)^{(1/p)_G} \geq 1
\end{equation}
 for all $x\in C_p^G (\Delta_G^m)$.\par
 Next we show that $\lVert x \lVert_p^G = 1 \Leftrightarrow x = (1, 1, \dots)$. Let $x = (1, 1, \dots)$. Then clearly $\lVert x \lVert_p^G = 1$. Conversely let $x = (x_k)\in C_p^G (\Delta_G^m)$ be such that $\lVert x \lVert_p^G = 1$. Then
\begin{equation}\label{equalto1}
\leftidx{_G}\sum\limits_{i=1}^{m} \lvert x_i \rvert_G \oplus \left(\leftidx{_G}\sum\limits_{n=1}^{\infty}\left\lvert (e\oslash e^n)\odot \leftidx{_G}\sum\limits_{k=1}^{n}\Delta_G^m x_k \right\rvert_G ^ {p_G}\right)^{(1/p)_G} = 1.
\end{equation}
Using 2.1(ii), \eqref{twopartsgreaterthan1} and \eqref{equalto1}, we have
\begin{equation*}
\leftidx{_G}\sum\limits_{i=1}^{m} \lvert x_i \rvert_G =1 \ \text{and} \ \left(\leftidx{_G}\sum\limits_{n=1}^{\infty}\left\lvert (e\oslash e^n)\odot \leftidx{_G}\sum\limits_{k=1}^{n}\Delta_G^m x_k \right\rvert_G ^ {p_G}\right)^{(1/p)_G} = 1.
\end{equation*}
That is,
\begin{align*}
&\lvert x_1 \rvert_G \dots \lvert x_m \rvert_G =1\\
\text{and} \ &\left\lvert (e\oslash e^1)\odot \leftidx{_G}\sum\limits_{k=1}^{1}\Delta_G^m x_k \right\rvert_G ^ {p_G}\cdot \left\lvert (e\oslash e^2)\odot \leftidx{_G}\sum\limits_{k=1}^{2}\Delta_G^m x_k \right\rvert_G ^ {p_G}\dots =1.
\end{align*}
Since the left hand side of the above equalities are products of the terms having magnitude either greater than or equal to one, we have
\begin{equation*}
 \lvert x_i \rvert_G =1 \ \text{for all} \ i = 1, 2, \dots, m \
\text{and} \ \left\lvert (e\oslash e^n)\odot \leftidx{_G}\sum\limits_{k=1}^{n}\Delta_G^m x_k \right\rvert_G =1 \ \text{for all} \ n.
\end{equation*}
That is,
\begin{align}
&x_i = 1 \ \text{for all} \ i = 1, 2, \dots , m \label{equation_xiequal1}\\
\text{and} \ \ & \leftidx{_G}\sum\limits_{k=1}^{n}\Delta_G^m x_k = 1 \ \text{for all} \ n.\label{equation_2ndpartequal1}
\end{align}
Taking $n=1$ in \eqref{equation_2ndpartequal1}, we get
\begin{equation}\label{equation_binomequal1}
\Delta_G^m x_1 = \leftidx{_G}\sum\limits_{v=0}^{m}(\ominus e)^{v_G}\odot e^{\binom mv}\odot x_{1+v} = 1.
\end{equation}
Putting the values of $x_1, \dots, x_m$ from \eqref{equation_xiequal1} in \eqref{equation_binomequal1}, we have $x_{m+1} = 1$. Similarly if we take $n=2, 3, \dots$ successively in \eqref{equation_2ndpartequal1}, then we get $x_{m+2} = x_{m+3} = \dots = 1$. Thus $x = (x_k) = (1, 1, \dots)$.\\
Next we show that $\lVert a\odot x \rvert_p^G = \lvert a \rvert_G \odot \lVert x \rVert_p^G$ for all $a\in \mathbb{R}(G)$ and $x \in C_p^G(\Delta_G^m)$. We consider
\begin{equation*}
\lVert a\odot x \rVert_p^G = \leftidx{_G}\sum\limits_{i=1}^{m} \lvert a\odot x_i \rvert_G \oplus \left(  \leftidx{_G}\sum\limits_{n=1}^{\infty} \left\lvert(e\oslash e^n) \odot \leftidx{_G}\sum\limits_{k=1}^{n}\Delta_G^m (a\odot x_k)\right\rvert_G^{p_G}\right)^{(1/p)_G}
\end{equation*}
Since $(\mathbb{R}(G), \oplus,\odot)$ is a field and the operator $\Delta_G^m$ is linear, we get
\begin{equation*}
\lVert a\odot x \rVert_p^G = \leftidx{_G}\sum\limits_{i=1}^{m} \lvert a\odot x_i \rvert_G \oplus \left(  \leftidx{_G}\sum\limits_{n=1}^{\infty} \left\lvert a\odot(e\oslash e^n) \odot \leftidx{_G}\sum\limits_{k=1}^{n}\Delta_G^m x_k\right\rvert_G^{p_G}\right)^{(1/p)_G}
\end{equation*}
From the property of geometric modulus, we get
\begin{align*}
\lVert a\odot x \rVert_p^G &= \left(\leftidx{_G}\sum\limits_{i=1}^{m} \lvert a \rvert_G \odot\lvert x_i \rvert_G \right)\oplus \left(  \leftidx{_G}\sum\limits_{n=1}^{\infty} \lvert a \rvert _G^{p_G}\odot\left\lvert (e\oslash e^n) \odot \leftidx{_G}\sum\limits_{k=1}^{n}\Delta_G^m x_k\right\rvert_G^{p_G}\right)^{(1/p)_G}\\
& = \lvert a \rvert_G \odot \left\{\leftidx{_G}\sum\limits_{i=1}^{m} \lvert x_i \rvert_G \oplus \left(  \leftidx{_G}\sum\limits_{n=1}^{\infty} \left\lvert(e\oslash e^n) \odot \leftidx{_G}\sum\limits_{k=1}^{n}\Delta_G^m x_k\right\rvert_G^{p_G}\right)^{(1/p)_G}\right\}\\
&= \lvert a \rvert_G \odot \lVert x \rVert_p^G .
\end{align*}
Finally we show geometric subaddivity in $C_p^G(\Delta_G^m)$. Consider
\begin{equation*}
\lVert x \oplus y \rVert_p^G = \leftidx{_G}\sum\limits_{i=1}^{m} \lvert x_i \oplus y_i\rvert_G \oplus \left(  \leftidx{_G}\sum\limits_{n=1}^{\infty} \left\lvert(e\oslash e^n) \odot \leftidx{_G}\sum\limits_{k=1}^{n}\Delta_G^m (x_k \oplus y_k)\right\rvert_G^{p_G}\right)^{(1/p)_G}.
\end{equation*}
From geometric triangular inequality and geometric Minkowski's inequality, we have
\begin{align*}
\lVert x\oplus y \rVert_p^G
  &\leq \left(\leftidx{_G}\sum\limits_{i=1}^{m} \lvert x_i\rvert_G \oplus \leftidx{_G}\sum\limits_{i=1}^{m} \lvert y_i\rvert_G\right) \oplus \Bigg\{\left(  \leftidx{_G}\sum\limits_{n=1}^{\infty} \left\lvert(e\oslash e^n) \odot \leftidx{_G}\sum\limits_{k=1}^{n}\Delta_G^m x_k \right\rvert_G^{p_G}\right)^{(1/p)_G}\\
  &\qquad \oplus\left(  \leftidx{_G}\sum\limits_{n=1}^{\infty} \left\lvert(e\oslash e^n) \odot \leftidx{_G}\sum\limits_{k=1}^{n}\Delta_G^m y_k \right\rvert_G^{p_G}\right)^{(1/p)_G}\Bigg\}\\
  &\leq \lVert x \rVert_p^G \oplus \lVert y \rVert_p^G.
\end{align*}
Thus $C_p^G(\Delta_G^m)$ is a normed linear space. Similarly following the similar lines, one can prove that $C_{\infty}^G(\Delta_G^m)$ is also a normed %linear space.
\end{proof}

%====================================================================================================================

\begin{theorem}
The normed spaces $C_p^G(\Delta_G^m)$ and $C_{\infty}^G(\Delta_G^m)$ are Banach spaces with respect to the norms \eqref{norm1} and \eqref{norm2} respectively.
\end{theorem}
\begin{proof}
We prove the theorem for the space $C_{\infty}^G(\Delta_G^m)$ only because the proof for the space $C_p^G(\Delta_G^p)$ runs along the same line. Let $(x^r)$ be a Cauchy sequence in $C_{\infty}^G(\Delta_G^m)$, where $x^r = (x_1^r, x_2^r, \dots) \in C_{\infty}^G(\Delta_G^m)$ for each $r = 1, 2, \dots$. Then,
\begin{equation*}
\lVert x^r \ominus x^t\lVert_{\infty}^G \xrightarrow[]{\text{G}} 1
\end{equation*}
as $r$ and $t$ tends to $\infty$. That is,
\begin{equation*}
\leftidx{_G}\sum\limits_{i=1}^{m} \lvert x_i^r \ominus x_i^t\rvert_G \oplus \sup\limits_{n}\left\lvert (e\oslash e^n)\odot \leftidx{_G}\sum\limits_{k=1}^{n}\Delta_G^m (x_k^r \ominus x_k^t)\right\rvert_G \xrightarrow[]{\text{G}} 1
\end{equation*}
as $r$ and $t$ tends to $\infty$. This implies that
\begin{equation*}
\leftidx{_G}\sum\limits_{i=1}^{m} \lvert x_i^r \ominus x_i^t \rvert_G  \xrightarrow[]{\text{G}} 1 \ \text{as well as} \ \sup\limits_{n}\left\lvert (e\oslash e^n)\odot \leftidx{_G}\sum\limits_{k=1}^{n}\Delta_G^m (x_k^r \ominus x_k^t)\right\rvert_G \xrightarrow[]{\text{G}} 1
\end{equation*}
as $r$ and $t$ tends to $\infty$ because of the property 2.1(ix). Consequently,
\begin{equation}\label{tend_completenessmod1}
\lvert x_i^r \ominus x_i^t \rvert_G \xrightarrow[]{\text{G}} 1
\end{equation}
for $i = 1, 2, \dots, m$ and
\begin{equation}\label{tend_completenessmod2}
\left\lvert (e\oslash e^n)\odot \leftidx{_G}\sum\limits_{k=1}^{n}\Delta_G^m (x_k^r \ominus x_k^t)\right\rvert_G \xrightarrow[]{\text{G}} 1
\end{equation}
for all $n \in \mathbb{N}$ when $r$ and $t$ tends to $\infty$. Putting $n = 1, 2, \dots$ in (\ref{tend_completenessmod2}) and applying (\ref{tend_completenessmod1}), we get
\begin{equation*}
\lvert x_k^r \ominus x_k^t\rvert_G \xrightarrow[]{\text{G}} 1
\end{equation*}
for each $k$ when $r$ and $t$ tends to infinity. Thus, for each $k$ the sequence $(x_k^r)_{r=1}^{\infty}$ is a Cauchy sequence in $\mathbb{R}(G)$. Since $\mathbb{R}(G)$ is complete, the sequence $(x_k^r)_{r=1}^{\infty}$ converges, that is $x_k^r \xrightarrow[]{\text{G}} x_k$ (say) for each $k$ as $r$ tends to infinity. As $(x^r)$ is a Cauchy sequence, there exists a natural number $N$ for each $\epsilon > 1$ such that
\begin{equation*}
\lVert x^r \ominus x^t\lVert_{\infty}^G < \epsilon
\end{equation*}
for all $r, t \geq N$. Hence,
\begin{equation}\label{inequality_completenessepsilon1}
\leftidx{_G}\sum\limits_{i=1}^{m} \lvert x_i^r \ominus x_i^t \rvert_G < \epsilon  \ \text{and} \ \sup\limits_n\left\lvert (e\oslash e^n)\odot \leftidx{_G}\sum\limits_{k=1}^{n}\Delta_G^m (x_k^r \ominus x_k^t)\right\rvert_G < \epsilon
\end{equation}
for all $r, t \geq N$. Fix $r$ and let $t$ tends to infinity in (\ref{inequality_completenessepsilon1}), we get
\begin{equation}\label{inequality_completenessepsilon2}
\leftidx{_G}\sum\limits_{i=1}^{m} \lvert x_i^r \ominus x_i \rvert_G < \epsilon  \ \text{and} \ \sup\limits_n\left\lvert (e\oslash e^n)\odot \leftidx{_G}\sum\limits_{k=1}^{n}\Delta_G^m (x_k^r \ominus x_k)\right\rvert_G < \epsilon
\end{equation}
for all $r \geq N$. This shows that
\begin{equation*}
\lVert x^r \ominus x\rVert_{\infty}^G < \epsilon^2
\end{equation*}
for all $r \geq N$. Thus, the sequence $(x^r)$ converges to the sequence $x=(x_k)$. Now, we need to show that the sequence $x=(x_k) \in C_{\infty}^G(\Delta_G^m)$. For this, we consider
\begin{align*}
&\left\lvert (e \oslash e^n)\odot \leftidx{_G}\sum\limits_{k=1}^{n}\Delta_G^m x_k \right\rvert_G\\
&=\left\lvert (e \oslash e^n)\odot \leftidx{_G}\sum\limits_{k=1}^{n}\Delta_G^m (x_k\oplus x_k^N \ominus x_k^N) \right\rvert_G \\
&\leq\left\lvert (e \oslash e^n)\odot \leftidx{_G}\sum\limits_{k=1}^{n}\Delta_G^m x_k^N \right\rvert_G \oplus \left\lvert (e \oslash e^n)\odot \leftidx{_G}\sum\limits_{k=1}^{n}\Delta_G^m (x_k^N\ominus x_k )\right\rvert_G.
\end{align*}
From the Inequalities (\ref{inequality_completenessepsilon2}) and keeping in view that the sequence $(x^N) \in C_{\infty}^G(\Delta_G^m)$, we conclude that $(x_k) \in C_{\infty}^G(\Delta_G^m)$. Therefore, the space $C_{\infty}^G(\Delta_G^m)$ is a Banach space.
\end{proof}
%===========================================================================================================

\begin{lemma}\label{lemma_pqG}
Let $(a_k)$ be a sequence in $\mathbb{R}(G)$ and $0< p < q$, then the following inequality holds:
\begin{equation*}
\left( \leftidx{_G}\sum\limits_{k=1}^{n}\lvert a_k\rvert_G^{q_G}\right)^{(1/q)_G} \leq \left( \leftidx{_G}\sum\limits_{k=1}^{n}\lvert a_k\rvert_G^{p_G}\right)^{(1/p)_G}.
\end{equation*}
\end{lemma}
\begin{proof}
Let
\begin{equation*}
u = \left( \leftidx{_G}\sum\limits_{k=1}^{n}\lvert a_k\rvert_G^{q_G}\right)^{(1/q)_G}.
\end{equation*}
Converting the above equation in classical arithmetic, we get
\begin{equation*}
\ln u =  \left( \sum\limits_{k=1}^{n}  \lvert  \ln a_k\rvert^q \right)^{(1/q)}.
\end{equation*}
\allowdisplaybreaks
By Lemma \ref{lemma_pq}, we have
\begin{equation*}
\ln u \leq   \left( \sum\limits_{k=1}^{n} \lvert  \ln a_k\rvert^p\right)^{(1/p)}.
\end{equation*}
Now converting this inequality in geometric arithmetic, we get
\begin{equation*}
u \leq \left( \leftidx{_G}\sum\limits_{k=1}^{n}\lvert a_k\rvert_G^{p_G}\right)^{(1/p)_G}.
\end{equation*}
\allowdisplaybreaks
This proves the Lemma.
\end{proof}
We now prove some inclusion relations.
%=================================================================================================================

\begin{theorem}
If $1\leq p < q< \infty$, then the inclusion $C_p^G(\Delta_G^m) \subset C_q^G(\Delta_G^m)$ holds.
\end{theorem}
\begin{proof}
The proof of this theorem easily follows from Lemma \ref{lemma_pqG}. So, we omit it.
\end{proof}
%========================================================================================================

\begin{theorem}
If $1 \leq p < \infty$, then the inclusion $C_p^G(\Delta_G^{m-1}) \subset  C_p^G(\Delta_G^{m})$ holds strictly.
\end{theorem}
\begin{proof}
Let a sequence $x = (x_k) \in C_p^G(\Delta_G^{m-1})$. Now we consider
\begin{align*}
&\left\lvert (e\oslash e^n)\odot \leftidx{_G}\sum\limits_{k=1}^{n} \Delta_G^m x_k \right\rvert_G \\
&= \left\lvert (e\oslash e^n)\odot \leftidx{_G}\sum\limits_{k=1}^{n} \Delta_G^{m-1} (x_k \ominus x_{k+1}) \right\rvert_G\\
&=\left\lvert (e\oslash e^n)\odot \leftidx{_G}\sum\limits_{k=1}^{n} \Delta_G^{m-1}x_k  \ominus (e\oslash e^n)\odot \leftidx{_G}\sum\limits_{k=1}^{n} \Delta_G^{m-1}x_{k+1}\right\rvert_G.\\
%&\leq \left\lvert (e\oslash e^n)\odot \leftidx{_G}\sum\limits_{k=1}^{n} \Delta_G^{m-1} x_k \right\rvert_G \oplus \left\lvert (e\oslash e^n)\odot \leftidx{_G}\sum\limits_{k=1}^{n} \Delta_G^{m-1} x_{k+1} \right\rvert_G.
\end{align*}
The triangular inequality suggests that
\begin{align*}
&\left\lvert (e\oslash e^n)\odot \leftidx{_G}\sum\limits_{k=1}^{n} \Delta_G^m x_k \right\rvert_G \\
&\leq \left\lvert (e\oslash e^n)\odot \leftidx{_G}\sum\limits_{k=1}^{n} \Delta_G^{m-1} x_k \right\rvert_G \oplus \left\lvert (e\oslash e^n)\odot \leftidx{_G}\sum\limits_{k=1}^{n} \Delta_G^{m-1} x_{k+1} \right\rvert_G.
\end{align*}
From Lemma \ref{lemma_maddoxG}, we have
\begin{align*}
&\left\lvert (e\oslash e^n)\odot \leftidx{_G}\sum\limits_{k=1}^{n} \Delta_G^m x_k \right\rvert_G^{p_G} \\
&\leq M\left\{\left\lvert (e\oslash e^n)\odot \leftidx{_G}\sum\limits_{k=1}^{n} \Delta_G^{m-1} x_k \right\rvert_G^{p_G} \oplus \left\lvert (e\oslash e^n)\odot \leftidx{_G}\sum\limits_{k=1}^{n} \Delta_G^{m-1} x_{k+1} \right\rvert_G^{p_G}\right\},
\end{align*}
where $M = e^{2^{p-1}}$. Taking geometric summation from $n = 1$ to $s$, we get
\begin{align*}
&\leftidx{_G}\sum\limits_{n = 1}^{s}\left\lvert (e\oslash e^n)\odot \leftidx{_G}\sum\limits_{k=1}^{n} \Delta_G^m x_k \right\rvert_G^{p_G} \\
&\leq M\left\{\leftidx{_G}\sum\limits_{n = 1}^{s}\left\lvert (e\oslash e^n)\odot \leftidx{_G}\sum\limits_{k=1}^{n} \Delta_G^{m-1} x_k \right\rvert_G^{p_G} \oplus \leftidx{_G}\sum\limits_{n = 1}^{s}\left\lvert (e\oslash e^n)\odot \leftidx{_G}\sum\limits_{k=1}^{n} \Delta_G^{m-1} x_{k+1} \right\rvert_G^{p_G}\right\}.
\end{align*}
As $s \rightarrow \infty$, we obtain
\begin{align*}
&\leftidx{_G}\sum\limits_{n = 1}^{\infty}\left\lvert (e\oslash e^n)\odot \leftidx{_G}\sum\limits_{k=1}^{n} \Delta_G^m x_k \right\rvert_G^{p_G} \\
&\leq M\left\{\leftidx{_G}\sum\limits_{n = 1}^{\infty}\left\lvert (e\oslash e^n)\odot \leftidx{_G}\sum\limits_{k=1}^{n} \Delta_G^{m-1} x_k \right\rvert_G^{p_G} \oplus \leftidx{_G}\sum\limits_{n = 1}^{\infty}\left\lvert (e\oslash e^n)\odot \leftidx{_G}\sum\limits_{k=1}^{n} \Delta_G^{m-1} x_{k+1} \right\rvert_G^{p_G}\right\} \\
&< \infty.
\end{align*}
Hence, the inclusion $C_p^G(\Delta_G^{m-1}) \subset  C_p^G(\Delta_G^{m})$ holds. To show strictness of the inclusion, we consider the sequence $x = (e^{k^{m-1}})$.
Then
\begin{equation*}
\Delta_G^m x_k = \leftidx{_G}\sum\limits_{v=0}^{m} (\ominus e)^{v_G}\odot e^{\binom mv}\odot e^{(k+v)^{m-1}}.
\end{equation*}
Converting the above equation into classical arithmetic, we get
\begin{eqnarray*}
\Delta_G^m x_k &= &e^{\sum\limits_{v=0}^{m}(-1)^v {\binom mv} {(k+v)}^{m-1}}\\
                                 &=&e^{\Delta^m k^{m-1}} = e^0 = 1.
\end{eqnarray*}
Then
\begin{equation}\label{equation_example1a}
\leftidx{_G}\sum\limits_{n=1}^{\infty}\left \lvert (e\oslash e^n)\odot  \leftidx{_G}\sum\limits_{k=1}^{n}\Delta_G^m x_k\right\rvert_G^{p_G} =1 < \infty.
\end{equation}
This shows that $x=(e^{k^{m-1}})\in C_p^G(\Delta_G^m)$. Now we show that $x= (e^{k^{m-1}})\notin C_p^G(\Delta_G^{m-1})$. Converting $\Delta_G^{m-1}x_k$ into classical arithmetic, we get
\begin{equation*}
\Delta_G^{m-1}x_k = e^{\Delta^{m-1}k^{m-1}} = e^{(-1)^{m-1} (m-1)!}.
\end{equation*}
Then
\allowdisplaybreaks
\begin{align}
\leftidx{_G}\sum\limits_{n=1}^{\infty}\left \lvert (e\oslash e^n)\odot  \leftidx{_G}\sum\limits_{k=1}^{n}\Delta_G^{m-1} x_k\right\rvert_G^{p_G} &= \leftidx{_G}\sum\limits_{n=1}^{\infty}\left \lvert e^{1/n}\odot  \leftidx{_G}\sum\limits_{k=1}^{n}e^{(-1)^{m-1} (m-1)!}\right\rvert_G^{p_G}\nonumber\\
&= \leftidx{_G}\sum\limits_{n=1}^{\infty}\left \lvert e^{1/n}\odot e^{(-1)^{m-1} (m-1)!\,n}\right\rvert_G^{p_G}\nonumber\\
&= \leftidx{_G}\sum\limits_{n=1}^{\infty}\left \lvert  e^{\frac{1}{n}(-1)^{m-1} (m-1)!\,n}\right\rvert_G^{p_G}\nonumber\\
&= \left \lvert  e^{(-1)^{m-1} (m-1)!}\right\rvert_G^{p_G}\odot\leftidx{_G}\sum\limits_{n=1}^{\infty}e\nonumber\\
&= \left \lvert  e^{(-1)^{m-1} (m-1)!}\right\rvert_G^{p_G}\odot e^{\sum 1} \rightarrow \infty.\label{equation_example1b}
\end{align}
This implies that $x= (e^{k^{m-1}})\notin C_p^G(\Delta_G^{m-1})$. Thus $x = (e^{k^{m-1}})$ belongs to $C_p^G(\Delta_G^{m})$ but does not belong to $C_p^G(\Delta_G^{m-1})$. Hence the inclusion is strict.
\end{proof}
Similarly, the inclusion $C_{\infty}^G(\Delta_G^{m-1}) \subset  C_{\infty}^G(\Delta_G^{m})$ also holds strictly and strictness can be seen by considering the sequence $x = (e^{k^m})$ that belongs to $C_{\infty}^G(\Delta_G^{m})$ but does not belong to $C_{\infty}^G(\Delta_G^{m-1})$.
%======================================================================================================

\section{Dual spaces and matrix transformations}
In this section, we determine $\alpha$-dual of the space  $C_{\infty}^G(\Delta_G^m)$ and study some matrix transformations. The $\alpha$- and $\beta$- duals of a sequence space $X$ in bigeometric calculus are denoted by $X^{\alpha}$ and $X^{\beta}$ and defined as
\begin{equation*}
X^{\alpha} = \left\{ a = (a_k) \in w(G): \leftidx{_G}\sum\limits_{k=1}^{\infty}\lvert a_k \odot x_k\rvert_G < \infty \ \text{for all} \ (x_k) \in X\right\}
\end{equation*}
and
\begin{equation*}
X^{\beta} = \left\{ a = (a_k) \in w(G): \leftidx{_G}\sum\limits_{k=1}^{\infty} a_k \odot x_k \ \text{converges for all} \ (x_k) \in X\right\}
\end{equation*}
respectively. We note that if two spaces $X$ and $Y$ are such that $X \subseteq Y$, then $Y^{\alpha} \subseteq X^{\alpha}$ and $Y^{\beta} \subseteq X^{\beta}$.
For $X = C_p^G(\Delta_G^m)$ or $C_{\infty}^G(\Delta_G^m)$, we define an operator $\upsilon : X \rightarrow X$ by $\upsilon(x) = (1, \dots , 1, x_{m+1}, x_{m+2}, \dots)$ for all $x = (x_k) \in X$. Consider the sets $\upsilon C_p^G(\Delta_G^m)$ and $\upsilon C_{\infty}^G(\Delta_G^m)$ as follows:
\begin{equation*}
\upsilon C_p^G(\Delta_G^m) = \left\{ x = (x_k) : x \in C_p^G(\Delta_G^m) \ \text{and} \ x_1 = x_2 = \dots = x_m = 1\right\}
\end{equation*}
and
\begin{equation*}
\upsilon C_{\infty}^G(\Delta_G^m) = \left\{ x = (x_k) : x \in C_{\infty}^G(\Delta_G^m) \ \text{and} \ x_1 = x_2 = \dots = x_m = 1\right\}.
\end{equation*}
%=================================================================================================================

\begin{lemma}\label{lemma_self1}
If $x = (x_k) \in \upsilon C_{\infty}^G(\Delta_G^m)$, then $\sup\limits_{k} e^{k^{-1}}\odot \lvert \Delta_G^{m-1}x_k\rvert_G < \infty$.
\end{lemma}
\begin{proof}
Let $x = (x_k) \in \upsilon C_{\infty}^G(\Delta_G^m)$, then
\begin{equation}\label{inequality_definitionmuinfinity}
\sup\limits_n \left\lvert (e\oslash e^n)\odot \leftidx{_G}\sum\limits_{k=1}^{n} \Delta_G^m x_k\right\rvert_G < \infty.
\end{equation}
Now, we consider
\begin{eqnarray*}
\leftidx{_G}\sum\limits_{k=1}^{n} \Delta_G^m x_k &=& \leftidx{_G}\sum\limits_{k=1}^{n} \Delta_G^{m-1}( x_k\ominus x_{k+1}) \\
                                                                                        &=& \leftidx{_G}\sum\limits_{k=1}^{n} \left( \Delta_G^{m-1} x_k \ominus \Delta_G^{m-1} x_{k+1}\right) \\
                                                                                        &=& \Delta_G^{m-1}x_1 \ominus \Delta_G^{m-1} x_{n+1}.
\end{eqnarray*}
As $x_1 = x_2 = \dots = x_m = 1$, we get $\Delta_G^{m-1}x_1 = 1$. Therefore,
\begin{equation*}
\leftidx{_G}\sum\limits_{k=1}^{n} \Delta_G^m x_k =\ominus \Delta_G^{m-1} x_{n+1}.
\end{equation*}
Putting this value in the Inequality (\ref{inequality_definitionmuinfinity}), we get
\begin{equation*}
\sup\limits_n \left\lvert (e\oslash e^n)\odot \Delta_G^{m-1}x_{n+1}\right\rvert_G < \infty.
\end{equation*}
Using the property 2.1(xi), we have
\begin{equation*}
\sup\limits_n \left\lvert e\oslash e^n\right\rvert_G \odot \left\lvert \Delta_G^{m-1}x_{n+1}\right\rvert_G < \infty.
\end{equation*}
Since $e\oslash e^n = e^{n^{-1}} > 1$, so we get
\begin{equation}\label{expression_supnenminus1}
\sup\limits_n e^{n^{-1}} \odot \left\lvert \Delta_G^{m-1}x_{n+1}\right\rvert_G < \infty.
\end{equation}
Now,
\begin{align}
\sup\limits_n e^{n^{-1}}\odot \left\lvert \Delta_G^{m-1}x_n\right\rvert_G & = \sup\limits_n e^{(1-\frac{1}{n})(n-1)^{-1}}\odot \left\lvert\Delta_G^{m-1}x_n\right\rvert_G\nonumber\\
&\leq \sup\limits_n e^{(n-1)^{-1}}\odot \left\lvert \Delta_G^{m-1}x_n\right\rvert_G.\label{inequality_nminusonenminusone}
\end{align}
From \eqref{expression_supnenminus1} and \eqref{inequality_nminusonenminusone}, we get
\begin{equation*}
\sup\limits_n e^{n^{-1}}\odot \left\lvert \Delta_G^{m-1}x_n\right\rvert_G< \infty.
\end{equation*}
Replacing $n$ by $k$, we get the result.
\end{proof}
%==================================================================================================================

\begin{lemma}\label{lemma_self2}
If a sequence $x=(x_k)\in \upsilon C_{\infty}^G(\Delta_G^m)$, then $\sup\limits_k e^{k^{-m}}\odot \lvert x_k \rvert_G < \infty$.
\end{lemma}
\begin{proof}
Let $x=(x_k)\in \upsilon C_{\infty}^G(\Delta_G^m)$, then from Lemma \ref{lemma_self1} $\sup\limits_{k} e^{k^{-1}}\odot \lvert \Delta_G^{m-1}x_k\rvert_G < \infty$, which then by Corollary \ref{corollary_khanekminus1} turns out to be $\sup\limits_k e^{k^{-m}}\odot\lvert x_k \rvert_G < \infty$.
\end{proof}
%%=================================================================================================================

\begin{theorem}
$\left[\upsilon C_{\infty}^G(\Delta_G^m)\right]^{\alpha} = \left\{ a=(a_k): \leftidx{_G}\sum\limits_{k=1}^{\infty} e^{k^m}\odot \lvert a_k \rvert_G < \infty \right\}$.
\end{theorem}
\begin{proof}
We consider $U = \left\{ a=(a_k): \leftidx{_G}\sum\limits_{k=1}^{\infty} e^{k^m}\odot \lvert a_k \rvert_G < \infty \right\}$.  Let $a = (a_k) \in U$; then for any $x = (x_k)\in \upsilon C_{\infty}^G(\Delta_G^m)$, we have
\begin{eqnarray*}
\leftidx{_G}\sum\limits_{k=1}^{\infty} \lvert a_k \odot x_k\rvert_G &=& \leftidx{_G}\sum\limits_{k=1}^{\infty} e^{k^{m}}\odot \lvert a_k \rvert_G \odot \left( e^{k^{-m}}\odot\lvert x_k \rvert_G\right)\\
&\leq& \sup\limits_k \left( e^{k^{-m}}\odot\lvert x_k \rvert_G\right) \odot  \leftidx{_G}\sum\limits_{k=1}^{\infty} e^{k^{m}}\odot \lvert a_k \rvert_G
\end{eqnarray*}
which is finite because of Lemma \ref{lemma_self2}. Since $x = (x_k)$ is arbitrary, we conclude that $a =(a_k) \in \left[\upsilon C_{\infty}^G(\Delta_G^m)\right]^{\alpha}$. Hence,
\begin{equation}\label{inclusion_upsiloninfinityalpha1}
U \subseteq \left[\upsilon C_{\infty}^G(\Delta_G^m)\right]^{\alpha}.
\end{equation}
Conversely, let $a \in [\upsilon C_{\infty}^G (\Delta_G^m)]^{\alpha}$. Then, $\leftidx{_G}\sum\limits_{k=1}^{\infty}\lvert a_k \odot x_k \rvert_G < \infty$ for each $x = (x_k) \in \upsilon C_{\infty}^G (\Delta_G^m)$. Now, consider the sequence $x = (x_k)$ that is defined by
\begin{equation}\label{sequence_upsiloninfinityalpha}
x_k = \begin{cases} 1, \ k \leq m\\ e^{k^m}, \ k > m\end{cases} .
\end{equation}
The sequence given by \eqref{sequence_upsiloninfinityalpha} belongs to $\upsilon C_{\infty}^G(\Delta_G^m)$. Hence, $\leftidx{_G}\sum\limits_{k = m+ 1}^{\infty}\left\lvert a_k \odot e^{k^m}\right\rvert_G < \infty$. Consequently,
\begin{equation*}
\leftidx{_G}\sum\limits_{k=1}^{\infty} \left\lvert e^{k^m} \odot a_k\right\rvert_G = \leftidx{_G}\sum\limits_{k=1}^{m} \left\lvert e^{k^m} \odot a_k\right\rvert_G\oplus \leftidx{_G}\sum\limits_{k=m+1}^{\infty} \left\lvert e^{k^m} \odot a_k\right\rvert_G < \infty.
\end{equation*}
This implies that $a \in U$. Thus,
\begin{equation}\label{inclusion_upsiloninfinityalpha2}
\left[\upsilon C_{\infty}^G(\Delta_G^m)\right]^{\alpha} \subseteq U.
\end{equation}
Inclusions \eqref{inclusion_upsiloninfinityalpha1} and \eqref{inclusion_upsiloninfinityalpha2} prove the theorem.
\end{proof}
%==================================================================================================================

\begin{theorem}
$\left[ C_{\infty}^G(\Delta_G^m) \right]^{\alpha}=\left[\upsilon C_{\infty}^G(\Delta_G^m) \right]^{\alpha}.$
\end{theorem}
\begin{proof}
Since $\upsilon C_{\infty}^G(\Delta_G^m)  \subset C_{\infty}^G(\Delta_G^m) $, we have $\left[ C_{\infty}^G(\Delta_G^m) \right]^{\alpha} \subset \left[\upsilon C_{\infty}^G(\Delta_G^m) \right]^{\alpha}$. Conversely, let $a = (a_k) \in \left[\upsilon C_{\infty}^G(\Delta_G^m) \right]^{\alpha}$, then $\leftidx{_G}\sum\limits_{k=1}^{\infty} \lvert a_k \odot x_k\rvert_G < \infty$ for all $x= (x_k) \in \upsilon C_{\infty}^G(\Delta_G^m)$. Now, consider any sequence $x' = (x'_k) \in C_{\infty}^G(\Delta_G^m)$, then the corresponding sequence $(1, 1, \dots , 1, x'_{m+1}, x'_{m+2}, \dots) \in \upsilon C_{\infty}^G(\Delta_G^m)$ and $\leftidx{_G}\sum\limits_{k = m+1}^{\infty} \lvert a_k \odot x'_k\rvert_G < \infty$. Thus,
\begin{equation*}
\leftidx{_G}\sum\limits_{k=1}^{\infty} \lvert a_k \odot x'_k\rvert_G = \leftidx{_G}\sum\limits_{k=1}^{m} \lvert a_k \odot x'_k\rvert_G \oplus \leftidx{_G}\sum\limits_{k=m+1}^{\infty} \lvert a_k \odot x'_k\rvert_G < \infty.
\end{equation*}
for all $x = (x'_k) \in C_{\infty}^G(\Delta_G^m)$. Therefore, the sequence $a = (a_k) \in \left[ C_{\infty}^G(\Delta_G^m) \right]^{\alpha}$ and $\left[\upsilon C_{\infty}^G(\Delta_G^m) \right]^{\alpha} \subseteq \left[ C_{\infty}^G(\Delta_G^m) \right]^{\alpha}$. Hence, the result.
\end{proof}
%=====================================================================================================================

Let us denote the spaces of bounded, convergent and absolutely $p$-summable sequences in bigeometric calculus by $l_{\infty}^G$, $c^G$ and $l_p^G$ respectively. Then the next result tells us that under certain conditions on the matrix $A$, which transforms $l_{\infty}^G$ or $c^G$  to $C_p^G(\Delta_G^m)$. We state the theorem as follows:
\begin{theorem}\label{theorem_matrix1}
Let $E = l_{\infty}^G$ or $c^G$ and $A = (a_{nk})$ be an infinite matrix whose entries are geometric real numbers, then $A \in (E, C_p^G(\Delta_G^m))$ for $1 \leq  p < \infty$ if and only if
\begin{enumerate}[(i)]
\item $\leftidx{_G} \sum\limits_{k=1}^{\infty} \lvert a_{nk} \rvert_G < \infty$ and
\item $B \in (E, l_p^G)$
\end{enumerate}
hold, where $B = (b_{ik}) = \left(e^{1/i}\odot (\Delta_G^{m-1}a_{1k} \ominus \Delta_G^{m-1}a_{i+1,k})  \right)$.
\end{theorem}
\begin{proof}
The sufficiency part of the theorem is trivial. To prove the necessity part, let us suppose that the matrix $A \in (E, C^G_p(\Delta_G^m))$ for $1 \leq p < \infty$. Then the series $A_n(x) = \leftidx{_G} \sum\limits_{k=1}^{\infty}a_{nk}\odot x_k$ converges for all $n$ and for all $x = (x_k) \in E$ and the sequence $(A_n(x)) \in C_p^G(\Delta_G^m)$. As the series $\leftidx{_G} \sum\limits_{k =1}^{\infty} a_{nk} \odot x_k$ converges for all $x  = (x_k) \in E$, the sequence $(a_{nk})_k \in E^{\beta} = l_{1}^G$. Thus the condition $(i)$ follows.  Since the sequence $(A_n(x)) \in C_p^G(\Delta_G^m)$, we get
\begin{align*}
&\leftidx{_G} \sum\limits_{i=1}^{\infty} \left \lvert e^{1/i} \odot \leftidx{_G} \sum\limits_{n=1}^{i} \Delta_G^m A_n (x)\right\rvert_G^{p_G} \\
&=\leftidx{_G} \sum\limits_{i=1}^{\infty} \left \lvert e^{1/i} \odot \leftidx{_G} \sum\limits_{n=1}^{i} (\Delta_G^{m-1} A_n (x)\ominus \Delta_G^{m-1}A_{n+1}(x))\right\rvert_G^{p_G} \\
&=\leftidx{_G} \sum\limits_{i=1}^{\infty} \left \lvert e^{1/i} \odot (\Delta_G^{m-1} A_1 (x)\ominus \Delta_G^{m-1}A_{i+1}(x))\right\rvert_G^{p_G} \\
&=  \leftidx{_G} \sum\limits_{i=1}^{\infty} \left \lvert  e^{1/i} \odot \leftidx{_G} \sum\limits_{k=1}^{\infty} (\Delta_G^{m-1}a_{1k} \ominus \Delta_G^{m-1}a_{i+1, k})\odot x_k\right \rvert_G^{p_G}\\
&= \leftidx{_G} \sum\limits_{i=1}^{\infty} \left\lvert \leftidx{_G} \sum\limits_{k=1}^{\infty} e^{1/i} \odot (\Delta_G^{m-1}a_{1k} \ominus \Delta_G^{m-1}a_{i+1, k})\odot x_k\right\rvert_G^{p_G} < \infty.
\end{align*}
Now if the matrix $B = (b_{ik})$ is such that  $b_{ik} = \left(e^{1/i}\odot (\Delta_G^{m-1}a_{1k} \ominus \Delta_G^{m-1}a_{i+1,k})  \right)$, then we have $\leftidx{_G} \sum\limits_{i=1}^{\infty} \left\lvert \leftidx{_G}\sum\limits_{k=1}^{\infty} b_{ik} \odot x_k\right \rvert_G^{p_G} < \infty$. Hence $B = (b_{ik}) \in (E, l_p^G)$, so the condition $(ii)$ follows.
\end{proof}
%====================================================================================================================

\begin{theorem}
Let $E = l_{\infty}^G$ or $c^G$ and $A = (a_{nk})$ be an infinite matrix whose entries are geometric real numbers, then $A \in (E, C_{\infty}^G(\Delta_G^m))$ if and only if
\begin{enumerate}[(i)]
\item $\leftidx{_G} \sum\limits_{k=1}^{\infty} \lvert a_{nk} \rvert_G < \infty$ and
\item $B \in (E, l_{\infty}^G)$
\end{enumerate}
hold, where $B = (b_{ik}) = \left(e^{1/i}\odot (\Delta_G^{m-1}a_{1k} \ominus \Delta_G^{m-1}a_{i+1,k})  \right)$.
\end{theorem}
\begin{proof}
The proof of this theorem runs along the similar lines as that of the Theorem \ref{theorem_matrix1}. So, we omit details of the proof.
\end{proof}
%====================================================================================================================

\section{Bigeometric Hermite interpolation}
In this section, we study how bigeometric calculus is useful to interpolate any function that is defined in $\mathbb{R}(G)$ and whose range lies in $\mathbb{R}(G)$. We give an interpolating formula in bigeometric calculus analogous to the Hermite interpolating formula in classical calculus. Let the values of a function $f$ and its derivative $f'$ are defined at $n+1$ distinct points $x_0, x_1, \dots , x_n$ on the interval $[a, b]$; then the Lagrange's form of classical Hermite polynomial is given by
\begin{equation*}
p(x) = \sum\limits_{i=0}^{n}Q_i(x)f(x_i) + \sum\limits_{i=0}^{n} {\hat{Q}}_i(x)f'(x_i),
\end{equation*}
where the polynomials
\begin{align*}
 Q_i(x) &= [1 - 2L'_{n,i}(x_i)(x-x_i)]L^2_{n,i}(x),\\
 {\hat{Q}}(x) &= (x-x_i)L^2_{n,i}(x)
 \end{align*}
  and the Lagrange's polynomials $L_{n,i}$ are defined by
  \begin{equation*}
  L_{n,i}(x) = \frac{(x - x_0) (x - x_1)\dots (x - x_{i-1}) (x - x_{i+1})\dots  (x - x_n)}{(x_i - x_0) (x_i - x_1)\dots (x_i -x_{i-1}) (x_i - x_{i+1})\dots  (x_i - x_n)}.
  \end{equation*}
  We derive an equivalent interpolating polynomial in bigeometric calculus.
\begin{theorem}
Let $f$ be a function such that $f(x_i)$ and $D_Gf(x_i)$ for $i = 0, 1, \dots , n$ are defined at each of the points $x_0, x_1, \dots , x_n$ in the geometric interval $[a, b]$. Then there is a unique bigeometric polynomial, $p_G$, of geometric degree at most $2n+1$ such that $p_G(x_i) = f(x_i)$ and $D_G\{p_G(x_i)\}=D_Gf(x_i)$ for each $i = 0, 1, \dots, n$.
\end{theorem}
\begin{proof}
 Define bigeometric polynomials $ H_{i} (x)$ and ${\hat{H}}_i(x)$ of geometric degree
$2n+1$ as
\begin{equation*}
H_i(x) = [e \ominus e^2 \odot D_G \{T_{n,i}(x_i)\}\odot (x\ominus x_i)] \odot{ T_{n,i}}^{2_G}(x)
\end{equation*}
and
\begin{equation*}
{\hat{H}}_i(x)= (x\ominus x_i)\odot T_{n,i}^{2_G}(x),
\end{equation*}
where $T_{n,i}(x)$ is defined by
\begin{equation*}
T_{n,i}(x) = \frac{(x\ominus x_0)\odot (x\ominus x_1)\odot\dots \odot(x\ominus x_{i-1})\odot (x\ominus x_{i+1})\odot\dots \odot (x\ominus x_n)}{(x_i\ominus x_0)\odot (x_i\ominus x_1)\odot\dots \odot(x_i\ominus x_{i-1})\odot (x_i\ominus x_{i+1})\odot\dots \odot (x_i\ominus x_n)}{}_{\tiny{G}}.
\end{equation*}
Clearly,
\begin{eqnarray*}
&H_i(x_j)=\begin{cases} e; \ i=j\\ 1 ;\  otherwise \end{cases} ,   &{\hat{H}}_i(x_j)=1 \ \text{\emph{for all}} \ i \ \text{\emph{and}} \ j, \\
&D_G\{H_i(x_j)\}=1 \ \text{\emph{for all}} \ i \ \text{\emph{and}} \ j, &D_G\{ {\hat{H}}_i(x_j)\}=\begin{cases} e; \ i=j\\ 1 ;\  otherwise \end{cases}.
\end{eqnarray*}
Now we consider the polynomial
\begin{equation*}
p_G(x) = \leftidx{_G}\sum\limits_{i=0}^{n} H_i(x) \odot f(x_i) \oplus \leftidx{_G}\sum\limits_{i=0}^{n} {\hat{H}}_i(x)\odot D_G f(x_i).
\end{equation*}
Then, $p_G(x_i) = f(x_i)$ and $D_G \{p_G(x_i)\} = D_G f(x_i)$. This proves the existence part of the theorem. To show uniqueness of the polynomial $p_G$, if possible suppose there is another polynomial $q_G$ of geometric multiplicity at most $2n+1$  such that $q_G \neq p_G$, $q_G(x_i) = f(x_i)$ and $D_G \{q_G(x_i)\} = D_G f(x_i)$ for $i= 0, 1, \dots , n$. Then the polynomial $r(x) = p_G(x) \ominus q_G(x)$ will be of geometric degree at most $2n+1$ with $r(x_i) =1$ and $D_G  r(x_i) = 1$. Thus, each $x_i$ is a geometric root of $r(x)$ with geometric multiplicity $2$. Therefore, $r(x)$ has $2n+2$ geometric roots, whereas its geometric degree is at most $2n+1$. This shows that $r(x)\equiv 1$. That is, both the polynomials $p_G$ and $q_G$ are equal. This proves the theorem.
\end{proof}
\textbf{Construction of Newton's form of bigeometric Hermite interpolation formula:}
 We construct Newton's form of bigeometric Hermite interpolation formula. Let $z_i = x_{\lfloor i/2 \rfloor}$, where $\lfloor . \rfloor$ denotes the greatest integer function as follows. That is, $z_0 = x_0, \ z_1 = x_0, \ z_2 = x_1, \ z_3 = x_1, \ z_4 = x_2, \ z_5 = x_2$ and so on. We define the Newton's form of bigeometric Hermite interpolation formula as follows:
\begin{align}
p_G(x) &= f(x_0) \oplus \leftidx{_G} \sum\limits_{k=1}^{2n+1} f_G[z_0, z_1, \dots ,z_k] \odot \leftidx{_G}\prod \limits_{i=0}^{k-1} (x\ominus z_i)\\
         &= f(x_0) \oplus f_G[x_0,x_0]\odot (x\ominus x_0) \oplus f_G[x_0, x_0, x_1]\odot (x\ominus x_0)^{2_G} \nonumber\\
         & \qquad \oplus f_G[x_0,x_0,x_1,x_1] \odot (x\ominus x_0)^{2_G}\odot (x\ominus x_1) \oplus \dots \nonumber\\
         & \qquad \oplus f_G[x_0,x_0,x_1,x_1, \dots , x_{n},x_{n}]\odot (x\ominus x_0)^{2_G}\odot(x \ominus x_1)^{2_G} \nonumber\\
         & \qquad \odot \dots \odot (x \ominus x_{n-1})^{2_G}\odot(x\ominus x_n),\nonumber
\end{align}
 where
 \begin{align*}
 f_G[z_i, z_{i+1}] &= D_G f(z_i), \ \text{when} \ z_i = z_{i+1} \\
 f_G[z_i, z_{i+1}] &= \frac{f(z_{i+1})\ominus f(z_i) }{z_{i+1}\ominus z_i} {}_{\tiny{G}} = \left[\frac{f(z_{i+1})}{f(z_i)}\right]^{\frac{1}{\ln(\frac{z_{i+1}}{z_i})}},  \ \text{when} \ z_i \neq z_{i+1}
 \end{align*}
 and
 \begin{align*}
 f_G[z_i, z_{i+1}, \dots, z_{i+k}] &=
  \frac{f_G[z_{i+1}, \dots, z_{i+k}] \ominus f_G[z_i, \dots, z_{i+k-1}]}{z_{i+k} \ominus z_{i}} {}_{\tiny{G}} \ \text{for} \ k\geq 2.
   \end{align*}
   We can prove this formula by considering the polynomial
   \begin{equation}\label{formula_newtonsform2}
   p_G(x) = f(x_0) \oplus \leftidx{_G} \sum\limits_{k=1}^{2n+1} A_k \odot \leftidx{_G}\prod \limits_{i=0}^{k-1} (x\ominus z_i)
   \end{equation}
   and its bigeometric derivative
    \begin{equation}\label{formula_newtonsform3}
   D_G \{p_G(x)\} =  \leftidx{_G} \sum\limits_{k=1}^{2n+1} A_k \odot D_G\left\{\leftidx{_G}\prod \limits_{i=0}^{k-1} (x\ominus z_i)\right\},
   \end{equation}
   where $A_k$ for $k = 1, 2, \dots, 2n+1$ are constants to be determined. Putting the values of $x_i$ for $i = 0, 1, 2, \dots ,n$ in the equations (\ref{formula_newtonsform2}) and (\ref{formula_newtonsform3}), we get $A_k = f_G[z_0, z_1, \dots ,z_k]$.\par
   %======================================================================================================================

   Next we illustrate below the construction of bigeometric Hermite interpolating polynomial by giving some examples..
   \begin{example}
Let the values of a function $f(x)$ and its derive $f'(x)$ are given as shown in the table.
\vspace{2mm}
\begin{center}
%\scalebox{0.7}{
\begin{tabular}{ c| c c c c c c c c}
\hline
  $x$ & $e$  & $e^2$      \\
 \hline
$f(x)$ & $e^2$ & $e^{4}$ \\
\hline
$ f'(x)$ & $2e$ & $2e^2$ \\
\hline
 \end{tabular}
%}
\end{center}
We will construct the bigeometric Hermite interpolating polynomial $p_G(x)$. We first compute the value of $D_G f(x)$ at the given points by using the formula $D_G f(x) = e^{x\frac{f'(x)}{f(x)}}$.
\vspace{1mm}
\begin{center}
%\scalebox{0.7}{
\begin{tabular}{ c| c c c c c c c c}
\hline
  $x$ &  $e$ & $e^2$    \\
 \hline
$f(x)$ & $e^2$ & $e^4$ \\
\hline
$D_G f(x)$ & $e^2$ & $e^2$ \\
\hline
 \end{tabular}
%}
\end{center}
\vspace{1mm}
 The divided difference table for bigeometric Hermite interpolation is as follows:
 \vspace{1mm}
 \begin{center}
 \begin{tabular}{c| c c c}
 \hline
 $x$ & $f(x)$ & 1st & 2nd\\
 \hline
 $x_0 = e$ & $e^2$ & {} & {}\\
 {} & {} & $e^2$ & {}\\
 $x_0 = e$ & $e^2$ & {} & 1\\
 {} & {} & $e^2$ & {}\\
 $x_1 = e^2$ & $e^4$ & {} & 1\\
 {} & {} & $e^2$ & {}\\
 $x_1 = e^2$ & $e^4$ & {} & {}\\
 \hline
 \end{tabular}
 \end{center}
 From Newton's form of bigeometric Hermite interpolation formula, we get
 \begin{align*}
 p_G(x) &= f(x_0) \oplus f_G[x_0, x_0]\odot (x\ominus x_0)\\
               & = e^2 \oplus e^2\odot (x\ominus e)\\
               & = e^2 e^{2\ln (x/e)}  = x^2.
 \end{align*}
 \end{example}
%===================================================================================================================

\begin{example}
Let the values of the function $f(x) = \ln x$ are given as shown in the table. The data in the following table have been taken from \cite{boruah2017application}.
\vspace{.5cm}
\begin{center}
%\scalebox{0.7}{
\begin{tabular}{ c| c c c c c c c c}
\hline
  $x$ &  3 & 6  &  12  & 24   \\
 \hline
$f(x)$ & 1.0986 & 1.7918 & 2.4849 & 3.1781\\
\hline
%$D_G f(x)$ & 2.4849 & 1.7474 & 1.4954 & 1.3698\\
%\hline
 \end{tabular}
%}
\end{center}
We will Plot the graph of bigeometric Hermite interpolating polynomial $p_G(x)$. We first compute the value of $D_G \ln(x)$ at the given points by using the formula $D_G \ln (x) = e^{\frac{1}{\ln (x)}}$.
\vspace{.5cm}
\begin{center}
%\scalebox{0.7}{
\begin{tabular}{ c| c c c c c c c c}
\hline
  $x$ &  3 & 6  &  12  & 24   \\
 \hline
$f(x)$ & 1.0986 & 1.7918 & 2.4849 & 3.1781\\
\hline
$D_G f(x)$ & 2.4849 & 1.7474 & 1.4954 & 1.3698\\
\hline
 \end{tabular}
%}
\end{center}
\vspace{.5cm}
 The divided difference table for bigeometric Hermite interpolation is as follows:
\vspace{5mm}
\begin{center}
%\scalebox{0.7}{
\begin{tabular}{ c| c c c c c c c c}
\hline
  $x$ &  $f(x)$ & 1st  & 2nd  & 3rd  & 4th  & 5th  & 6th  & 7th   \\
 \hline
 $x_0 = 3$ & 1.0986 &  \\
           {}      &       {}       &  2.4849\\
 $x_0 = 3$ & 1.0986 &           {}    & 0.7445\\
           {}       &      {}       &  2.0254  &  {}           & 1.1257 \\
$x_1= 6$   & 1.7918 &  {}             &  0.8082 & {}          & 0.9613            \\
 {}                 & {}            & 1.7474   &  {}            & 1.0658 & {}          & 1.0139\\
 $x_1 = 6$ & 1.7918 & {}              & 0.8828 & {}              &0.9799 & {}           & 0.9959\\
 {}                 & {}            &  1.6028  & {}             & 1.0362   & {}           & 1.0053 & {}            & 1.0014  \\
$x_2 = 12$& 2.4849 & {}              & 0.9048  & {}              & 0.9908 & {}           & 0.9987 \\
{}                  & {}            & 1.4954   &  {}            & 1.0230    &  {}            &    1.0026 \\
 $x_2 = 12$& 2.4849 & {}             & 0.9338  &  {}             &  0.9944\\
 {}                   & {}           & 1.4261  & {}             & 1.0150 \\
$x_3 = 24$  & 3.1781& {}             & 0.9435  \\
{}                    & {}            & 1.3698  \\
 $x_3 = 24$&  3.1781\\
 \hline
\end{tabular}
%}
\end{center}
 \vspace{5mm}
From Newton's form of bigeometric Hermite interpolation formula, we get
\small{
\begin{align*}
p_G(x) & = f(x_0) \oplus f_G[x_0,x_0]\odot (x\ominus x_0) \oplus f_G[x_0, x_0, x_1]\odot (x\ominus x_0)^{2_G} \oplus f_G[x_0,x_0,x_1,x_1] \\
         & \quad \odot (x\ominus x_0)^{2_G}\odot (x\ominus x_1) \oplus f_G[x_0, x_0, x_1, x_1, x_2]\odot (x\ominus x_0)^{2_G}\odot (x\ominus x_1)^{2_G} \oplus\\
         & \quad f_G[x_0, x_0, x_1, x_1, x_2, x_2]\odot (x\ominus x_0)^{2_G}\odot (x\ominus x_1)^{2_G}\odot (x\ominus x_2) \oplus\\
         &\quad  f_G[x_0, x_0, x_1, x_1, x_2, x_2, x_3]\odot (x\ominus x_0)^{2_G}\odot (x\ominus x_1)^{2_G}\odot (x\ominus x_2)^{2_G} \oplus\\
         &\quad f_G[x_0, x_0, x_1, x_1, x_2, x_2, x_3,x_3]\odot (x\ominus x_0)^{2_G}\odot (x\ominus x_1)^{2_G}\odot (x\ominus x_2)^{2_G} \odot (x \ominus x_3)\\
       & = 1.0986 \oplus 2.4849\odot (x\ominus 3) \oplus 0.7445\odot (x\ominus 3)^{2_G} \oplus 1.1257 \odot (x\ominus 3)^{2_G}\\
         & \quad \odot  (x\ominus 6)\oplus 0.9613\odot (x\ominus 3)^{2_G}\odot (x\ominus 6)^{2_G} \oplus1.0139\odot (x \ominus 3)^{2_G}\\
         & \quad \odot (x\ominus 6)^{2_G}\odot (x\ominus 12) \oplus0.9959\odot (x\ominus 3)^{2_G}\odot (x\ominus 6)^{2_G}\odot (x\ominus 12)^{2_G} \\
         &\quad  \oplus 1.0014\odot (x\ominus 3)^{2_G}\odot (x\ominus 6)^{2_G}\odot (x\ominus 12)^{2_G} \odot (x \ominus 24).
\end{align*}}
\begin{figure}[h]
\centering
\includegraphics[height=7cm, width=7cm,keepaspectratio]{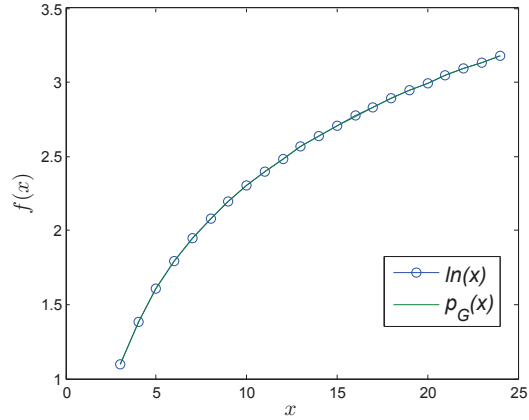}
\caption{Graph of $\ln(x)$ and $p_{\tiny{G}} (x)$ versus $x$.}\label{figure1}
\end{figure}
Figure \ref{figure1} shows that the graph of interpolating polynomial $p_G(x)$ almost overlapping the graph of $\ln (x)$ in the interval $[3, 24]$.
\end{example}
%===================================================================================================================

\section{Conclusion}
In this paper, we have introduced some sequence spaces in bigeometric calculus, determined their $\alpha$-duals and studied matrix transformations of these spaces. We have also derived an interpolating formula in bigeometric calculus and shown some related examples.

\section*{Acknowledgement}
The first author gratefully acknowledges a research fellowship awarded by the Council of Scientific and Industrial Research, Government of India (File No: 09/081(1246)/2015-EMR-I).

\bibliographystyle{plain}
\bibliography{mybib}

\begin{thebibliography}{10}

\bibitem{aniszewska2007multiplicative}
Dorota Aniszewska.
\newblock Multiplicative \text{R}unge--\text{K}utta methods.
\newblock {\em Nonlinear Dynamics}, 50(1):265--272, 2007.

\bibitem{aniszewska2005analysis}
Dorota Aniszewska and Marek Rybaczuk.
\newblock Analysis of the multiplicative \text{L}orenz system.
\newblock {\em Chaos, Solitons \& Fractals}, 25(1):79--90, 2005.

\bibitem{bashirov2008multiplicative}
Agamirza~E Bashirov, Emine~M{\i}s{\i}rl{\i} Kurp{\i}nar, and Ali
  {\"O}zyap{\i}c{\i}.
\newblock Multiplicative calculus and its applications.
\newblock {\em Journal of Mathematical Analysis and Applications},
  337(1):36--48, 2008.

\bibitem{boruah2017application}
Khirod Boruah and Bipan Hazarika.
\newblock Application of geometric calculus in numerical analysis and
  difference sequence spaces.
\newblock {\em Journal of Mathematical Analysis and Applications},
  449(2):1265--1285, 2017.

\bibitem{ccakmak2012some}
Ahmet~Faruk {\c{C}}akmak and Feyzi Ba{\c{s}}ar.
\newblock Some new results on sequence spaces with respect to
  non-\text{N}ewtonian calculus.
\newblock {\em Journal of Inequalities and Applications}, 2012:228, 2012.

\bibitem{ccakmak2015some}
Ahmet~Faruk {\c{C}}akmak and Feyzi Ba{\c{s}}ar.
\newblock Some sequence spaces and matrix transformations in multiplicative
  sense.
\newblock {\em TWMS Journal of Pure and Applied Mathematics}, 6(1):27--37,
  2015.

\bibitem{dovsenovic2016multiplicative}
Tatjana Do{\v{s}}enovi{\'c}, Mihai Postolache, and Stojan Radenovi{\'c}.
\newblock On multiplicative metric spaces: survey.
\newblock {\em Fixed Point Theory and Applications}, 2016:92, 2016.

\bibitem{et2000some}
Mikail Et.
\newblock On some generalized \text{C}es{\`a}ro difference sequence spaces.
\newblock {\em Istanbul University Fen Fak. Mathematics Dergisi},
  55--56:221--229, 1996--1997.

\bibitem{et1995some}
Mikail Et and R{\i}fat {\c{C}}olak.
\newblock On some generalized difference sequence spaces.
\newblock {\em Soochow Journal of Mathematics}, 21(4):377--386, 1995.

\bibitem{grossman1983bigeometric}
Michael Grossman.
\newblock {\em Bigeometric calculus: A system with a scale-free derivative}.
\newblock Archimedes Foundation, 1983.

\bibitem{grossman1972non}
Michael Grossman and Robert Katz.
\newblock {\em Non-Newtonian Calculus: A self-contained, elementary exposition
  of the authors' investigations}.
\newblock Lee Press, 1972.

\bibitem{kadak2014determination}
U{\u{g}}ur Kadak.
\newblock Determination of the \text{K}{\"o}the-\text{T}oeplitz duals over the
  non-\text{N}ewtonian complex field.
\newblock {\em The Scientific World Journal}, 2014, 2014.

\bibitem{kadak2016cesaro}
U{\u{g}}ur Kadak.
\newblock Ces{\`{a}}ro summable sequence spaces over the non-\text{N}ewtonian
  complex field.
\newblock {\em Journal of Probability and Statistics}, 2016, 2016.

\bibitem{kadak2016multiplicative}
U{\u{g}}ur Kadak.
\newblock On multiplicative difference sequence spaces and related dual
  properties.
\newblock {\em Boletim da Sociedade Paranaense de Matem{\'a}tica},
  35(3):181--193, 2016.

\bibitem{kadak2014construction}
U{\u{g}}ur Kadak and Hakan Efe.
\newblock The construction of \text{H}ilbert spaces over the
  non-\text{N}ewtonian field.
\newblock {\em International Journal of Analysis}, 2014, 2014.

\bibitem{kadak2015classical}
U{\u{g}}ur Kadak, Murat Kiri{\c{s}}ci, and Ahmet~Faruk {\c{C}}akmak.
\newblock On the classical paranormed sequence spaces and related duals over
  the non-\text{N}ewtonian complex field.
\newblock {\em Journal of Function Spaces}, 2015, 2015.

\bibitem{khan55generalized}
Shadab~Ahmad Khan and Ashfaque~A Ansari.
\newblock Generalized \text{K}{\"o}the-\text{T}oeplitz dual of some geometric
  difference sequence spaces.
\newblock {\em International Journal of Mathematics And its Applications},
  4(2):13--22, 2016.

\bibitem{kizmaz1981certain}
H~Kizmaz.
\newblock Certain sequence spaces.
\newblock {\em Canadian Mathematical Bulletin}, 24(2):169--176, 1981.

\bibitem{maddox1967spaces}
I~J Maddox.
\newblock Spaces of strongly summable sequences.
\newblock {\em The Quarterly Journal of Mathematics}, 18(1):345--355, 1967.

\bibitem{orhan1983cesaro}
C~Orhan.
\newblock Ces{\`a}ro difference sequence spaces and related matrix
  transformations.
\newblock {\em Comm. Fac. Univ. Ankara, Ser. A}, 32:55--63, 1983.

\bibitem{riza2009multiplicative}
Mustafa Riza, Ali {\"O}zyapici, and Emine Misirli.
\newblock Multiplicative finite difference methods.
\newblock {\em Quarterly of Applied Mathematics}, 67(4):745--754, 2009.

\bibitem{rybaczuk2000fractal}
M~Rybaczuk and P~Stoppel.
\newblock The fractal growth of fatigue defects in materials.
\newblock {\em International Journal of Fracture}, 103(1):71--94, 2000.

\bibitem{rybaczuk1999critical}
Marek Rybaczuk.
\newblock Critical growth of fractal patterns in biological systems.
\newblock {\em Acta of Bioengineering and Biomechanics}, 1(1):5--9, 1999.

\bibitem{turkmen2012some}
Cengiz T{\"u}rkmen and Feyzi Ba{\c{s}}ar.
\newblock Some basic results on the sets of sequences with geometric calculus.
\newblock {\em Commun. Fac. Sci. Univ. Ank. Series A1}, 61(2):17--34, 2012.

\end{thebibliography}

\end{document}